\documentclass[pdflatex,sn-chicago]{sn-jnl}

\usepackage{graphicx}%
\usepackage{multirow}%
\usepackage{amsmath,amssymb,amsfonts}%
\usepackage{amsthm}%
\usepackage{mathrsfs}%
\usepackage[title]{appendix}%
\usepackage{xcolor}%
\usepackage{textcomp}%
\usepackage{manyfoot}%
\usepackage{booktabs}%
\usepackage{algorithm}%
\usepackage{algorithmicx}%
\usepackage{algpseudocode}%
\usepackage{listings}%
\DeclareMathOperator*{\argmax}{arg\,max}

\theoremstyle{thmstyleone}%
\newtheorem{theorem}{Theorem}
\newtheorem{proposition}[theorem]{Proposition}%

\theoremstyle{thmstyletwo}%
\newtheorem{example}{Example}%

\theoremstyle{thmstylethree}%
\newtheorem{definition}{Definition}%

\begin{document}

\title[Sampling models for selective inference]{Sampling models for selective inference
}


\author*[1]{\fnm{Daniel} \sur{García Rasines}}\email{daniel.garciarasines@cunef.edu}
\equalcont{These authors contributed equally to this work.}

\author[2]{\fnm{G. Alastair} \sur{Young}}\email{alastair.young@imperial.ac.uk}
\equalcont{These authors contributed equally to this work.}

\affil*[1]{\orgdiv{Department of Quantitative Methods}, \orgname{CUNEF Universidad}, \orgaddress{\street{Calle Almansa 101}, \city{Madrid}, \postcode{28040}, \country{Spain}}. ORCID: 0000-0002-1558-5860}

\affil[2]{\orgdiv{Department of Mathematics}, \orgname{Imperial College London}, \orgaddress{\street{Exhibition Road}, \city{London}, \postcode{SW7 2AZ}, \country{United Kingdom}}. ORCID: 0000-0002-8333-7981}


\abstract{This paper explores the challenges of constructing suitable inferential models in scenarios where the parameter of interest is determined in light of the data, such as regression after variable selection. Two compelling arguments for conditioning converge in this context, whose interplay can introduce ambiguity in the choice of conditioning strategy: the Conditionality Principle, from classical statistics, and the `condition on selection' paradigm, central to selective inference. We discuss two general principles that can be employed to resolve this ambiguity in some recurrent contexts. The first one refers to the consideration of how information is processed at the selection stage. The second one concerns an exploration of ancillarity in the presence of selection. We demonstrate that certain notions of ancillarity are preserved after conditioning on the selection event, supporting the application of the Conditionality Principle. We illustrate these concepts through examples and provide guidance on the adequate inferential approach in some common scenarios.}

\keywords{Selection, inference, conditioning, ancillarity}



\maketitle

\section{Introduction}

Many statistical methods rely on the assumption that the inferential questions have been specified before analysing the data. However, actual statistical workflows tend to be more dynamic. Often, upon interacting with the data, the analyst determines which aspects of the data-generating process are most interesting to analyse, and then reuses the same dataset to study them. This is particularly relevant in modern science, where the size and complexity of datasets make it difficult to determine an appropriate set of inferential objectives beforehand. It is well known that this adaptivity renders classical statistical procedures invalid, typically leading to large estimation biases and uncertainty underestimation. This problem, commonly referred to as selective inference, has garnered considerable attention in the past few years. This paper is concerned with the appropriate choice of statistical model in this context.

To fix ideas, let us consider a standard problem concerning selection: the `selected mean' problem. Suppose that $Y_1, \ldots, Y_n$ are independent normal observations with different unknown means $\theta_1, \ldots, \theta_n$ and known common variance, and that we seek inference for the means of the largest $k$ observations for some $k < n$. A problem like this can arise, for example, in a clinical trial experiment, where $n$ treatments are tested and only the most promising ones are put forward for further investigation. In this situation, reporting standard estimates and confidence intervals for the selected parameters ignoring selection---a face-value approach---will likely lead to flawed conclusions, as the largest observations tend to overestimate their corresponding means. This particular case of selection bias is known in the literature as the winner's curse. Besides being important in its own right, this problem provides a simple theoretical benchmark for the analysis of more complex scenarios. 

A general formulation of selection problems is as follows. Suppose that we have data $Y\in \mathcal{Y}$, whose sampling distribution we model by some parametric family $\{F(y; \theta)\colon \theta\in \Theta \}$, where $F(y; \theta)$ is the distribution function of $Y$ under $\theta$. We assume throughout that $\mathcal{Y}$ and $\Theta$ are both subsets of a Euclidean space, and that all the distributions in the family are dominated by the Lebesgue or counting measures. The corresponding Lebesgue densities or probability mass functions will be denoted by $f(y; \theta)$. In addition, we assume that there exists a set of $m$ potential parameters of interest, $\{\psi_1(\theta), \ldots, \psi_m(\theta)\}$, from which at most one is to be selected for inference after observing the data, possibly by an artificially randomised procedure. For example, in the previous problem, the set of potential parameters of interest would contain all the subsets of $\{\theta_1, \ldots, \theta_n\}$ of size $k$. Thus, we assume that there exist functions $ p_i\colon \mathcal{Y} \to [0, 1]$, $i = 1, \ldots, m$, such that, having observed $Y = y$, $\psi_i(\theta)$ is selected for inference with probability $p_i(y) = \mathbb{P}(I = i\mid y)$, where $I$ denotes the random index of the selected parameter.

Under this formulation of the problem, at most one of the parameters can be selected, so when considering a specific parameter we will generally drop the subscript and write it simply as $\psi \equiv \psi_i(\theta)$, and the corresponding selection probability as $p(y) \equiv p_i(y)$. The latter function, which encodes all the relevant information regarding the selection mechanism of $\psi$, will be referred to as the selection function. By randomised selection procedure, we mean one for which the decision to select a parameter for inference depends not only on $y$, but also on a random element independent of $Y$, typically random noise generated by the statistician. In these cases, $p(y)$ is not an indicator function. This is typically carried out to increase the power of the subsequent inference \citep{tian}. The following examples illustrate the framework.

\begin{example} \label{EX: file_drawer}
\textit{File drawer problem}. Let $Y_1, \ldots, Y_{n_1} \in \mathbb{R}$ be a random sample from a distribution with scalar parameter $\theta$, and suppose that we use this sample to test $H_0\colon \theta = 0$ against $H_a\colon \theta > 0$, because we are only interested in $\theta$ if it is positive. If the test results in rejection, we collect a second set of samples $Y_{n_1 + 1}, \ldots, Y_n$ to obtain more information about $\theta$. At the end, the whole data vector is $Y = (Y_1, \ldots, Y_n)^T$. If the testing procedure rejects $H_0$ whenever $T(Y_1, \ldots, Y_{n_1}) > t$ for some statistic $T$ and some threshold $t$, the selection function is $p(y) = \mathbf{1}\{ T(y_1, \ldots, y_{n_1}) > t \}$, where $\mathbf{1}\{\cdot \}$ is the indicator function. Within this context, a randomised procedure could be of the form $T(y_1, \ldots, y_{n_1}) + W > t$, where $W$ is generated from a known distribution independently of the data, and the resulting selection function would be $p(y) = \mathbb{P}\{ T(y_1, \ldots, y_{n_1}) + W > t \}$.
\end{example}

Now, suppose that for a certain population under study the true parameter is $\theta = 0$, and that a process like the one described before is carried out by several independent researchers. If only those analyses for which the initial test rejected the null are published, the literature concerning this parameter will be filled with results that overstate the true effect size, leading to the false belief that $\theta$ is positive, when in fact it is not. This is a stylised example of what \cite{rosenthal} termed the `file drawer effect', also known as publication bias, which occurs when the decision of whether to make a scientific finding public is influenced by the outcome of the analysis. Another important family of selection problems concerns inference for a regression model after variable selection. This is discussed in greater depth in Section \ref{SEC: sampling}.

\begin{example} \label{EX: regression}
\textit{Regression model}. Let $Y \in \mathbb{R}^n$ be a response vector and $X\in \mathbb{R}^{n\times p}$ be a known, fixed design matrix containing the observed values of $p$ covariates. Suppose that a variable-selection algorithm is used to select a non-empty subset $s\subseteq \{1, \ldots, p\}$ of the covariates. We think of $s$ as being the output of some algorithm such as the LASSO or best subset selection applied to $(Y, X)$. Given $s$, we may wish to provide inference for the best linear predictor of $Y$ given only the selected covariates, $\psi = \{X(s)^T X(s)\}^{-1} X(s)^T \mathbb{E}_\theta[Y\mid X]$, where $X(s)$ is the submatrix of $X$ containing the observations of the selected covariates. Denoting by $E(s)  \subseteq \mathbb{R}^n$ the set of observations for which the variable-selection algorithm would have produced the same set $s$, we can write the selection function as $p(y) = \mathbf{1}\{y\in E(s)\}$.
\end{example}

Like many other fundamental problems in statistics, selection issues were first studied by R.A. Fisher in \cite{fisher34}, who considered the sampling bias arising in albinism studies. The type of problem considered by Fisher is one in which independent samples $Y_1, Y_2, \ldots$ from a given population are generated and each sample is recorded or not with probability $w(y_i)$, so that the data available for inference constitutes a random sample from a weighted distribution $ g(y;\theta) \propto f(y; \theta) w(y)$. There is a large body of literature dedicated to the study of these models. A good survey is given by \cite{rao85}. 

Although these problems fit within the framework considered here, on defining the selection function $p(y_1, \ldots, y_n) = \prod_{i = 1}^n w(y_i)$, they are different in nature to the kind of problems studied in this work. Conceptually, the main difference is that in our problems selection refers to the decision of providing or not inference for a given parameter, while in the other case the inferential objective is decided in advance, and selection refers to an inherent bias of the sampling mechanism. Mathematically, the selection functions considered here cannot usually be broken up into marginal selection functions for the different samples: the decision of providing inference for a given parameter depends on some summary statistic of the full dataset or a subset of it. The implications of this are important. When $p(y_1, \ldots, y_n) = \prod_{i = 1}^n w(y_i)$ and $Y_1, \ldots, Y_n$ are independent,  the distribution induced by the selection rule can be factorised into independent components. By contrast, the selection rules considered here induce non-trivial dependencies between the observations, and such dependencies do not vanish when the sample size increases except in trivial circumstances. 

Importantly, in this work we are only concerned with situations where the selection function of the selected parameter is known, at least to a very high degree of precision. Inference with a fully unknown selection function is ill-posed, as the  model becomes highly non-identifiable. This problem has been considered by some authors \citep[e.g.][]{bayarri-degroot} in cases where the selection function is assumed to be of the form $p(y) =\prod_{i = 1}^n w(y_i)$. 

\section{The conditional approach to selective inference}

Statistical inference for selected parameters constitutes a major source of discrepancy between the frequentist and Bayesian schools. Within the former framework, nominal error rates should ideally match (or at least not exceed) the actual errors reported on repeated use of a statistical procedure. Since inferences on a selected parameter are only reported for those samples that lead to selection of said parameter, errors should be assessed relative to hypothetical repetitions of the data-generating process that lead to consideration of the same parameter of interest. Thus, type-I and type-II errors, biases, coverages, and other measures of interest are computed conditionally on the output of the selection stage. We refer to the general principle of conditioning inferences on the selection event as the \textit{conditional approach} to selective inference. 

By contrast, the Bayesian stance is more contested. The classical argument asserts that, given that Bayesian inferences are fully conditioned on the data, they are, in particular, conditioned on the selection step. Therefore, no further adjustment is required \citep{dawid}. Two major criticisms have been voiced against this position. The first one, due to \cite{MR1, MR2} and \cite{yekutieli}, identifies two different ways in which conditioning on selection can affect the Bayesian sampling mechanism, leading to two correct but orthogonal treatments of selection. In the current setting, this distinction may be expressed as follows. Suppose, without loss of generality, that we observe data $y$ and select the first parameter, $\psi_1$. This situation is compatible with two conditional sampling mechanisms for $(\theta, Y)$:
\begin{itemize}
    \item \textit{Fixed-parameter regime}: $\theta$ is sampled from the prior $\pi(\theta)$, fixed, and values of $Y$ are sampled conditionally on $\theta$ until $\psi_1$ gets selected. In this case, the marginal distribution of $\theta$ is not affected by selection and the posterior density has
    \begin{equation*}
        \pi(\theta\mid y, I = 1) \propto \pi(\theta) \frac{ f(y; \theta)}{P(I = 1 \mid \theta)}.
    \end{equation*}
    \item \textit{Random-parameter regime}: pairs $(\theta, Y)$ are sampled jointly until $\psi_1$ gets selected. In this case, selection favours values of $\theta$ for which the probability of selecting $\psi_1$ is larger, and it is easily shown that
    \begin{equation*}
        \pi(\theta\mid y, I = 1) = \pi(\theta\mid y).
    \end{equation*}
    These are the cases where conditioning on selection is redundant.
\end{itemize}
Thus, from this viewpoint, a selection-adjustment of the posterior distribution is needed in the fixed-parameter regime but not in the random-parameter one. Additionally, \cite{handbook_open} observe that, even within the random-parameter case, failing to account for selection can yield inferential procedures with very poor repeated-sampling properties. For example, while the coverage of credible intervals is correct when averaged over the prior distribution, it can behave erratically when assessed conditionally on $\theta$ across a range of plausible values, a feature some Bayesians deem undesirable.

In this paper we adopt the frequentist position. According to the conditional approach, inference on the selected $\psi$ must be based on the conditional distribution of the data given the selection event, denoted by $S$. Henceforth, we refer to this distribution as the \textit{selective distribution}, which has density (or mass function)  
\begin{equation*} 
f(y\mid S; \theta)\equiv f_S(y; \theta ) = \frac{f(y; \theta) p(y)}{\varphi(\theta)}, \quad \varphi(\theta) = \mathbb{E}_\theta[p(Y)].
\end{equation*}
From a repeated-sampling perspective, if inferences are only provided for those samples for which $\psi$ gets selected, use of this sampling distribution for inference ensures that nominal error probabilities coincide with the reported long-term error rates.

In \cite{kuffneryoung} it is argued that the conditional approach is in fact a modern re-expression of well-established Fisherian arguments.  Specifically, the authors consider what they term the \textit{Fisherian proposition}, based on \cite{fisher1, fisher2}, which states that relevance of the inference to the observed data is achieved by conditioning on certain sample statistics. Fisher advocated conditioning on ancillary statistics, distribution constant statistics which therefore do not carry any information on the parameter of interest. This ensures not only relevance to the data, but also no loss of information via conditioning. Later authors, however, proposed conditioning on any statistic that captures evidential strength in the sample \citep{kiefer1, kiefer2, brown, berger}. In the selective context, a natural partitioning of the sample space that restores relevance to the realised data is provided by the selection algorithm. Conditioning on its output divides the sample space into regions that correspond to specification of the same inferential problem.

In a similar vein, we argue that the conditional approach follows from arguments analogous to those underpinning the Conditionality Principle \citep{birnbaum}, as both advocate conditioning on the random events that have occurred right up to the formulation of the inferential problem. However, the nature of the conditioning differs in each case. In the selective context, the conditioning statistic may carry information about the parameter of interest, as it contains evidence that the parameter is sufficiently relevant to warrant formal analysis. While this conditioning might discard information about $\psi$, it is necessary to avoid reusing any information already exploited to select this as the parameter of interest. In doing so, subsequent evidence about it is evaluated exclusively in light of independent evidence.

The conditional approach has received considerable attention in recent years, particularly in the context of inference after variable selection. Interest on this approach increased notably following the realization that some widely used selection rules produce simple selection events, like affine sets, which can be handled with relative ease. This line of work began with \cite{lockhartetal} and has been explored by numerous authors, including \cite{leetaylor, loftustaylor, leeetal, fithianetal, tibsetalboots}; and \cite{panigrahi-integrative}. Notable selection rules of this kind include the LASSO and LARS with fixed penalty parameters, stepwise procedures with a fixed number of steps, and multiplicity correction methods such as the Benjamini-Hochberg procedure \citep{BH}. On the other hand, progress on the conditional approach has been hindered by the analytical complexity of many contemporary selection algorithms. This has led to the development of analytical approximations \citep{panigrahi-scalable}, simulation-based approaches \citep{blackbox}, and randomisation mechanisms \citep{tian, splitting, leiner, witten-thinning}. 

\section{Sampling models} \label{SEC: sampling}

The conditional approach appears straightforward to apply from a conceptual perspective: if the selection function $p(y)$ is known and one is comfortable with a particular sampling model for the data, then the selective distribution is unequivocally determined. However, in this section we argue that, just as the determination of an appropriate sampling distribution in classical statistics is subject to some degree of ambiguity, so is, in some circumstances, the appropriate choice of the conditional distribution for selective inference. To motivate the discussion, let us consider the following example.

\begin{example} \label{EX: random_N}
\textit{Random sample size}. Let $Y_1, \ldots, Y_{n_1}$ be a random sample from a $N(\theta, 1)$ distribution, where the sample size $n_1\geq 1$ has been generated from some known distribution $f_{N_1}(n_1)$. Suppose that we use this sample to test whether $\theta$ is greater than zero. If we observe $y_1 + \ldots  + y_{n_1} > 1.96 n_1^{1/2}$, say, we collect a second set of observations of fixed size $n_2$, $Y_{n_1 + 1}, \ldots, Y_n$, where $n = n_1 + n_2$, which we use to obtain more information about $\theta$. In the joint model for $(N_1, Y_1, \ldots, Y_n)$, with selection function $p(n_1, y_1, \ldots, y_n) = \mathbf{1}(y_1 + \ldots  + y_{n_1} > 1.96 n_1^{1/2})$, the conditional distribution of the full data given selection is
\begin{equation*}
f_S(n_1, y_1, \ldots, y_n; \theta) = \frac{f_{N_1}(n_1) f(y_1, \ldots, y_n\mid n_1 ; \theta) p(n_1, y_1, \ldots, y_n)}{\sum_{\tilde n_1 = 1}^\infty f_{N_1}(\tilde n_1) \Phi(\theta \tilde n_1^{1/2} - 1.96)}.
\end{equation*}
Note that in this model the sample size is not independent of $\theta$, as its probability mass function is given by
\begin{equation*}
f_S(n_1;\theta ) = \frac{f_{N_1}(n_1) \Phi(\theta n_1^{1/2} - 1.96)}{\sum_{\tilde n_1 = 1}^\infty f_{N_1}(\tilde n_1) \Phi(\theta \tilde n_1^{1/2} - 1.96)}.
\end{equation*}
The intuitive explanation of this dependence is that, if the true $\theta$ is smaller than zero, we have a higher chance of falsely concluding that $\theta > 0$ if we have a less informative sample (i.e. if $n_1$ is small), and vice versa. Therefore, inference based on this sampling model would make use of the sample size distribution, and would arrive to a different conclusion had the sample size been decided in advance by the statistician rather than random, which appears counter-intuitive. Instead, it is more reasonable to work with the model
\begin{equation*} \label{EQ: sample size}
f_S(n_1, y_1, \ldots, y_n ; \theta) = \frac{f_{N_1}(n_1) f(y_1, \ldots, y_n\mid n_1 ; \theta) p(n_1, y_1, \ldots, y_n)}{\Phi(\theta n_1^{1/2} - 1.96)}.
\end{equation*}
This model conditions on selection \textit{after} the sample size has been observed, which therefore remains independent of the parameter after conditioning. This reflects more closely the selection process, since the selection test was designed to achieve a certain significance level conditionally on the observed $n_1$. That is, the information about $\theta$ in the selection step is interpreted according to the conditional model $Y_1, \ldots, Y_{N_1}\mid N_1$, rather than to the unconditional one.
\end{example}

Example \ref{EX: random_N} illustrates a key issue about selection models: conditioning on selection can create artificial dependencies in the data that do not reflect the nature of the sampling process. Formally, a statistic $A$ which is distribution-constant unconditionally will typically depend on the parameter after conditioning on selection. However, it is clear that whether a sample is selected or not cannot make $A$ informative. To avoid this, we have to ensure that the selective distribution reflects the selection process as closely as possible.  

Suppose that, unconditionally, the data $Y$ can be reduced to a minimal sufficient statistic $(T, A)$, where $A$ is distribution constant. In such cases $A$ is said to be ancillary for $\theta$, and the Conditionality Principle, a formal re-expression of the Fisherian proposition, asserts that the information provided by $Y$ about $\theta$ is equivalent to that provided by $T$ given the observed value of $A$ \citep{birnbaum}. If the selection process complies with this principle, so that the information about $\theta$ in the selection stage is interpreted via the conditional distribution $T\mid A$, as in Example \ref{EX: random_N}, then it is more reasonable to define the selective distribution by
\begin{equation*}
f_S(t, a; \theta) = \frac{f_A(a) f_{T\mid A}(t\mid a; \theta) p(t, a)}{\varphi(\theta; a)}
\end{equation*}
rather than
\begin{equation*}
f_S(t, a; \theta) = \frac{f_A(a) f_{T\mid A}(t\mid a; \theta) p(t, a)}{\varphi(\theta)},
\end{equation*}
where $p(t, a) = \mathbb{E}[p(Y)\mid t, a]$ denotes the selection function in terms of $(T, A) = (t, a)$, independent of $\theta$ by sufficiency, $f_A(a)$ and $f_{T\mid A}(t\mid a;\theta)$ are the densities or probability mass functions of $A$ and $T\mid A$, and $\varphi(\theta; a) = \mathbb{E}_\theta[p(T, A)\mid a]$ is the selection probability under $\theta$ given $A = a$. Ancillary statistics arise, for example, in models with an underlying group structure, such as location models.

\begin{example}
\textit{Location model}. Let $Y_1, \ldots, Y_n$ be a random sample from a location model with density $f(y_i; \theta) = g(y_i - \theta)$, $\theta\in \mathbb{R}$. In this model, the Conditionality Principle dictates that sample evidence about $\theta$ should be interpreted via the conditional distribution of $T = \hat\theta$ given $A = (Y_1 - \hat\theta, \ldots, Y_{n} - \hat{\theta})$, where $\hat{\theta}$ is the maximum likelihood estimator of $\theta$ and $A$ is the configuration statistic. The density of $\hat{\theta}\mid a$ admits the simple expression
\begin{equation*}
f_{T\mid A}(t\mid a; \theta) = c(\theta, a)\prod_{i = 1}^n g(t + a_i - \theta),
\end{equation*}
where $c(\theta, a)$ is a normalising constant. Suppose that $\theta$ is selected for inference if the $p$-value
\begin{equation*}
u(T, a) = c(0, a)\int_{\hat\theta}^\infty \prod_{i = 1}^n g(\tilde t + a_i) \mathrm{d}\tilde t
\end{equation*}
is below some level $\alpha$. Then, the selective density of $(\hat{\theta}, A)$ should be
\begin{equation*}
f_S(t, a; \theta) = \frac{f(a)c(\theta, a)\prod_{i = 1}^n g(t + a_i - \theta) p(t, a)}{\mathbb{P}_\theta(u(T, a) \leq \alpha\mid a)},
\end{equation*}
where $p(t, a) = \mathbf{1}\{ u(t, a)\leq \alpha \}$.
\end{example}

A similar case can be made about models involving nuisance parameters. Suppose that $\theta = (\psi, \chi)$, where $\psi$ is the potential parameter of interest and $\chi$ is a nuisance parameter. In these models it is sometimes possible to identify a minimal sufficient statistic $(T, A)$ such that the distribution of $T\mid A$ is independent of $\chi$ and the distribution of $A$ is independent of $\psi$, or depends on it in such a way that through observation of $A$ alone no information can be extracted about $\psi$ \citep[Chapter 2]{bncox}. In such cases, a natural extension of the Conditionality Principle asserts that information about $\psi$ should be interpreted via the conditional model $T\mid A$. As before, if the selection process interprets the sample information about $\psi$ via the conditional distribution $T\mid A$, it is more natural to define the selective density as 
\begin{equation*}
f_S(t, a; \theta) = \frac{f_A(a;\theta) f_{T\mid A}(t\mid a; \psi) p(t, a)}{\varphi(\psi; a)}
\end{equation*}
rather than
\begin{equation*}
f_S(t, a; \theta) = \frac{f_A(a;\theta) f_{T\mid A}(t\mid a; \psi) p(t, a)}{\varphi(\psi, \chi)},
\end{equation*}
where we are using the same notational conventions as before.

An important setting where this argument applies is when inference is sought for a component of the canonical parameter in a full exponential family. Let the model be given by
\begin{equation*}
f(y; \theta) = h(y) \exp\left\{ \psi s_1(y) + \chi^T s_2(y) - K(\theta) \right\}.
\end{equation*}
Standard arguments show that the distribution of $s_1(Y) \mid s_2(Y) $ is independent of $\chi$, and even though the distribution of $s_2(Y)$ usually depends on $\psi$, power considerations based on Neyman-Pearson theory indicate that inference for $\psi$ should be provided via the conditional distribution $s_1(Y) \mid s_2(Y) $ \cite[][Chapter 7]{youngsmith}. Moreover, $s_2(Y)$ is uninformative for $\psi$ in a sense made precise by \cite{jorgensen}, so the Conditionality Principle applies. 

A second important class of examples with this structure that arise frequently in selective inference is the following.

\begin{example} \label{EX: winners}
\textit{Inference on winners}. Suppose that $Y_1, \ldots, Y_m$ are independent observations with densities $f(y_i; \theta_i)$, $i = 1, \ldots, m$, and suppose that we want to provide inference for the parameter which produces the largest observation, $\theta_I = \argmax\{y_i\colon i = 1, \ldots, m\}$. Assume without loss of generality that $I = 1$, so that $T = Y_1$ and $A = (Y_2, \ldots, Y_m)^T$. We can identify two natural sampling mechanisms that are consistent with selection of the first mean. In the first one, the whole vector $(Y_1, \ldots, Y_m)^T$ is sampled until $Y_1$ is observed to be the maximum. The corresponding density of this generative process is
\begin{equation} \label{EQ: selmean1}
f_S(y_1, y_2\ldots, y_m; \theta) = \frac{\prod_{i=1}^m f(y_i; \theta_i) \mathbf{1}(I = 1)}{\mathbb{P}_\theta(Y_1 > Y_i \text{ }\forall \text{ } i>1)}.
\end{equation}
The second sampling mechanism is that in which $(Y_2, \ldots, Y_m)^T$ is sampled from its unconditional distribution, and conditionally on its observed value, $Y_1$ is sampled until it exceeds $\max\{ y_2, \ldots, y_m\}$. The density of the data under this model is
\begin{equation} \label{EQ: selmean2}
f_S(y_1, y_2\ldots, y_m; \theta) = \frac{\prod_{i=1}^m f(y_i; \theta_i) \mathbf{1}(I = 1)}{\mathbb{P}_{\theta_1}(Y_1 > Y_i \text{ }\forall \text{ } i>1\mid y_2, \ldots, y_m)}.
\end{equation}
\end{example}

In the first model, conditioning on selection breaks the independence structure of the data. As a consequence, the observations from the non-selected parameters depend on $\theta_1$, and the marginal distribution of $Y_1$ depends on $\theta_2, \ldots, \theta_m$. This makes manipulation of the likelihood function awkward and computationally expensive if $m$ is large, despite the apparent simplicity of the problem. However, one could nevertheless argue that it is appropriate to condition on the observed values of $Y_2, \ldots, Y_m$ for either of the following two reasons:
\begin{enumerate}
    \item If selection is used to determine whether $\theta_1$ is the largest $\theta_i$, since the observed values of $Y_2, \ldots, Y_m$ do not provide any direct information about this fact, it appears more reasonable to conduct the selection step conditionally on $a = (y_2, \ldots, y_m)^T$. Thus, preference of the second inferential model over the first can be justified without resorting to computational considerations.
    \item Even if the previous argument does not apply in a given problem, we might still decide to condition on $a$ on top of the selection event on the basis that the conditional distribution of $A$ given the selection event depends very mildly on the parameter of interest (it is non-informative, in a certain sense), and leads to simpler inferences. We formalise this notion in the following section. 
\end{enumerate}

Since the distribution of $Y_2, \ldots, Y_m$ depends on $\theta_1$ given selection, conditioning on them, even if advisable on conceptual grounds, is expected to reduce inferential power. We have investigated this empirically in \cite{handbook}, in a situation where the marginal distributions of the $Y_i$'s are Gaussian; see also Section 3 of \cite{handbook_open}. The main conclusion is that inference based on \eqref{EQ: selmean1} is slightly more powerful than inference based on the simpler model \eqref{EQ: selmean2} (as indicated by confidence intervals 5-10\% shorter). However, it requires considerably more computational power than the latter and is more sensitive to numerical approximations.

Another important area of application of selective inference involving ancillary statistics is regression after variable selection, which was briefly introduced in Example \ref{EX: regression}. In this setting, data consists of an $n$-dimensional response vector $Y$ and a $n\times p$ design matrix $X$, containing information on $p$ covariates. If $p$ large, either absolutely or relative to $n$, it is often convenient to reduce the set of covariates to be analysed. This reduction may be conducted to increase the interpretability of the fitted model, for computational or analytical efficiency, or because it is believed that only a few covariates have explanatory power over $Y$, and by getting rid of the inactive ones higher inferential power can be achieved. Following the notation of Example \ref{EX: regression}, let $s\subseteq \{1, \ldots, p\}$ be the subset of selected covariates and $X(s)$ be the corresponding sub-design matrix.

An ongoing matter of discussion in this context concerns the adequacy of conditioning on $X$ in cases where it has been randomly generated. In standard regression settings, a fixed-design approach can be compellingly justified by appealing to the Conditionality Principle, provided it is reasonable to assume that the distribution of $X$ carries no information about the regression parameter $\mu = E(Y\mid X)$. By contrast, incorporating a selection step into the analysis introduces additional complexity that must be accounted for. We distinguish the following two scenarios, which lead to two different treatments of the design matrix.

\textbf{Scenario 1}. Sometimes the goal of the selection step is to decide which joint distribution from $\{(Y, X(s))\colon s\subseteq \{1, \ldots, p\}\}$ is going to be analysed. That is, the selected parameter of interest is $\psi_s = E[Y\mid X(s)]$. In these circumstances, conditioning on $X$ before selection is not appropriate, as this would prevent consideration of the different joint distributions. Once we have selected the distribution for inference, the conditional approach dictates that the regression of $y$ on $X(s)$ ought to be carried out conditionally on the selection event $E = \{ (y, X) \colon s(y, X) = s\}$, where $s(y, X)$ is the output of the selection algorithm applied to a generic $(y, X)$. In general, the conditional distribution of $X(s)\mid E$ depends on the regression parameter $\psi_s$, so $X(s)$ would not be ancillary for it and a fixed-design analysis would not be justified. Conditioning on $E$ with a random $X$ is very challenging, as it requires modeling its distribution, and to the best of our knowledge the only viable proposal to achieve validity in this setting is data splitting \citep{rinaldo}.

\textbf{Scenario 2}. A different type of selection problem is as follows. In some cases, we may want to assess the effect of the selected covariates $X(s)$ relative to the full set of covariates $X$. In this context, standard arguments justify conditioning on $X$ before carrying out selection, thereby leading to a fixed-design approach. The selection event is then $E = \{ y \colon s(y, X) = s\}$, where $X$ is fixed at its observed value. This problem has been extensively studied and there are multiple analytical and computational procedures available to solve it, depending on the selection algorithm and the error distribution of $Y\mid X$ \citep{leetaylor, loftustaylor, leeetal, fithianetal, tibsetalboots, panigrahi-integrative, panigrahi-scalable, splitting, blackbox}. Typically, the parameter of interest in this context is the so-called projection parameter $\psi_s = [X(s)^TX(s)]^{-1}X(s)^T\mu$, which provides the best linear predictor of $Y$ using only the variables in $s$, though alternative parameters have been proposed.

\section{Ancillarity in selection models with nuisance parameters}

In the previous section, we argued that the choice of inferential model should account for whether the information available at the selection stage was analysed conditionally on an ancillary statistic. When this is not the case, statistics that were independent of the selected parameter before selection may become dependent on it after conditioning. This raises the key question of whether conditioning on selection actually renders these statistics informative about the parameter of interest.

The most intuitive definition of ancillarity in the presence of nuisance parameters requires the marginal distribution of the ancillary statistic $A$ to be free of $\psi$, which as we have just seen does not usually hold in selective distributions. Weaker definitions of ancillarity have been suggested that allow $A$ to depend on $\psi$, but in such a way that if we only observed $A$ we would not be able to extract any information about $\psi$. In this section we are going to show that two of these proposed notions of ancillarity are preserved after conditioning on selection, and can thus be invoked to justify conditioning on these statistics in situations not covered by the discussion in Section \ref{SEC: sampling}.

Throughout this section we will assume that we can write $\theta = (\psi, \chi)$, that the parameter space can be written as $\Theta = \Theta_1 \times \Theta_2$, where $\Theta_1$ and $\Theta_2$ are the parameter spaces of $\psi$ and $\chi$, and that there exists a minimal sufficient statistic $(T, A)$ for $\theta$ such that the distribution of $T\mid A$ does not depend on $\chi$. Instead of modifying the selective sampling distribution according to the nature of the sampling process, as in the previous section, we now assume that the selective distribution is $f_S(y; \theta)\propto f(y; \theta) p(y)$, and justify the change of inferential distribution differently.

In the previous setting, \cite{Godambe1} calls $A$ ancillary if it is complete for $\chi$ for any fixed value of $\psi$. This is, if, for all $\psi$, whenever $\mathbb{E}_{\psi, \chi}[g(A)] = 0$ for all $\chi \in \Theta_2$ and some real-valued measurable function $g$, it must be the case that $\mathbb{P}_{\psi, \chi}\{g(A) = 0\} = 1$ for all $\chi\in \Theta_2$. We will refer to this type of ancillarity as $G$-ancillary. \cite{Godambe2} showed that if $A$ is $G$-ancillary, then the distribution of $T$ contains the same information about $\psi$ as the distribution of $T\mid A$, so that the latter can be used for inference about $\psi$ without any loss of information. It is easy to see that $G$-ancillarity is preserved after conditioning on selection events under minimal assumptions. We have the following result.
\begin{proposition}\label{PROP: godambe}
If the selection mechanism is such that $\varphi(\psi; a) > 0$ for all possible $(\psi, a)$, then $A$ is $G$-ancillary in the non-selective model if and only if it is $G$-ancillary in the selective model.
\end{proposition} 
\begin{proof}
First of all, we have that, for any $y_1, y_2$ in the support of the selective distribution,
\begin{equation*}
\frac{f_S(y_1 ; \theta )}{f_S(y_2; \theta )} = \frac{f(y_1; \theta )}{f(y_2; \theta )}\frac{p(y_1)}{p(y_2)} 
\end{equation*}
is free of $\theta$ if and only if $f(y_1; \theta )/f(y_2; \theta )$ is free of $\theta$, which, by the minimal sufficiency of $(T, A)$ in the non-selective regime, holds if and only if $(T, A)$ is equal for $y_1$ and $y_2$. This implies that $(T, A)$ is also minimal sufficient for $\theta$ in the selective model. Secondly, we need to show that the distribution of $T\mid A$ in the selective model is free of $\chi$. This is trivial, as
\begin{equation*}\label{EQ: reducedconditionalmodel2}
f_{S}(t \mid  a ; \theta ) = \frac{p(t, a)}{\varphi(\psi;a)}   f(t\mid a; \psi).
\end{equation*}
Now, for $A$ to be $G$-ancillary in the non-selective model it means that the condition $\mathbb{E}_{\psi, \chi}[g(A)] = 0$ for all $\chi$ implies that $\mathbb{P}_{\psi, \chi}\{g(A) = 0\} = 1$ for all $\chi$, given an arbitrary $\psi$. However, note that 
\begin{equation*}
\mathbb{E}_{\psi, \chi}[g(A)\mid S] = \frac{\mathbb{E}_{\psi, \chi}[g(A)\varphi(\psi; A)]}{\varphi(\psi, \chi)},
\end{equation*}
so if $\mathbb{E}_{\psi, \chi}[g(A)\mid S] = 0$ we must have that $\mathbb{P}_{\psi, \chi}\{ g(A)\varphi(\psi; A) = 0\} = 1$. With the assumption that $\varphi(\psi; a) $ is positive everywhere, this implies that $\mathbb{P}_{\psi, \chi}\{ g(A) = 0\} = 1$, so $A$ is $G$-ancillary in the selective model. The reciprocal claim follows by the same type of argument.
\end{proof}

The second notion of ancillarity that we will consider stems from the field of noninformation theory, introduced by Barndorff-Nielsen in \cite{bn2} and extended by \cite{jorgensen}. Several types of ancillarity can be derived from this work. In the selective inference context, the most relevant one is $M$-ancillarity and extensions thereof, all of which follow from the following general definition: if, for every possible pair $(\psi, a)$, the model $\{f_A(a; \psi, \chi) \colon \chi \in \Theta_2\}$ provides a \textit{perfect fit} for the data $a$, then $A$ is said to be ancillary for $\psi$ \citep{jorgensen}. Different types of ancillarity can then be derived by adopting different definitions of perfect fit. 

The rationale behind this definition is as follows: if two models provide a perfect fit for a given observation, then both are equally good at explaining the data, and we cannot assume that one is more likely to have generated the data than the other. Thus, in the previous setting, we cannot compare two values of $\psi$, $\psi_1$ and $\psi_2$, say, only through observation of $A = a$, because the models $\{f_A(a; \psi_1, \chi) \colon \chi \in \Theta_2\}$ and $\{f_A(a; \psi_2, \chi) \colon \chi \in \Theta_2\}$ explain observation of $a$ equally well.

In the case of $M$-ancillarity, perfect fit is defined as follows \citep{bn2}. A family of densities or mass functions $\{g(x; \lambda)\colon \lambda \in \Lambda\}$ is said to provide a perfect fit for a given observation $x\in \mathcal{X}$ if there exists a $\lambda\in \Lambda$ such that $x$ is an approximate mode of $g(x; \lambda)$. Formally, if
\begin{equation} \label{EQ: M_noninformation}
\text{for all } \varepsilon > 0 \text{ there exists a } \lambda\in \Lambda \text{ such that } g(x; \lambda) > (1 - \varepsilon) \sup_{\tilde x\in \mathcal{X}} g(\tilde x; \lambda).
\end{equation}

One drawback of this definition is that it is not invariant under transformations of the data. For example, the family of exponential distributions does not provide a perfect fit for any positive observation, as all the densities have mode $0$. However, one can easily show that after applying a log-transformation the family provides a perfect fit for any real observation. To remove this arbitrariness, we propose to change the definition and require instead that the previous condition holds under some change of variable. This produces a new type of ancillarity, which we term $\tilde{M}$-ancillarity. We will assume that the distribution of $A$ is continuous for simplicity. Extension to the discrete case is straightforward.

\begin{definition} \label{DEF: tilde_M}
Let $\mathcal{A}$ be the sample space of $A$ and assume it is open. $A$ is $\tilde M$-ancillary for $\psi$ if, for all $\psi\in \Theta_1$, there exists a smooth bijection $h$ such that $\{f_{h(A)}(z; \psi, \chi) \colon \chi \in \Theta_2\}$ provides a perfect fit for any $z\in h(\mathcal{A})$ in the sense of Equation \eqref{EQ: M_noninformation}.
\end{definition}

We then have the following proposition.
\begin{proposition}
If the selection mechanism is such that $\varphi(\psi; a) > 0$ for all possible $(\psi, a)$, then $A$ is $\tilde M$-ancillary in the non-selective model if and only of it is $\tilde M$-ancillary in the selective model. 
\end{proposition}

\begin{proof}
The first part of the proof is the same as in Proposition \ref{PROP: godambe}. The marginal density of $A$ given selection is
\begin{equation*}
f_{S}(a;\theta) = \frac{\varphi(\psi;a)}{\varphi(\theta)} f_A(a;\theta).
\end{equation*}
Suppose that $A$ is $\tilde M$-ancillary for $\psi$ in the non-selective model. Let $\psi$ be arbitrarily fixed and define $Z_1 = h_1(A)$, and $Z_2 = h_2(Z_1)$, where $h_1 $ is a bijection such that $Z_1$ satisfies the condition of Definition \ref{DEF: tilde_M}, and $h_2$ is a bijection with absolute Jacobian determinant $\varphi(\psi;h_1^{-1}(z_1))\vert Jh_1^{-1}(z_1) \vert$, where $\vert J h_1^{-1}(z_1) \vert$ denotes the absolute Jacobian determinant of $h_1^{-1}(z_1)$. By construction, the density of $Z_2$ given selection is proportional to $f_{Z_1}(h_2^{-1}(z_1); \psi, \chi)$. By assumption, for any $z_1$ there exists a $\chi$, $\hat\chi(z_1)$, such that $z_1$ is a mode of $f_{Z_1}(\cdot; \psi, \hat\chi(z_1))$. Hence, for all $z_2$, $\hat\chi\{h_2^{-1}(z_2)\}$ is such that $z_2$ is a mode of $f_{S, Z_2}(\cdot; \psi, \hat\chi\{h_2^{-1}(z_2)\})$, which shows $\tilde M$-ancillarity of $A$ in the selective model. 
The converse can be proved analogously.
\end{proof}

The two forms of ancillarity considered in this section can be used to formally justify conditioning on certain sample statistics which depend on the parameter of interest. The simplest scenario is that described in Example \ref{EX: winners}, where the non-selective distribution of the data is the product of $m$ independent random variables with different underlying parameters. In that example, if the data from the non-selected parameters, $A = (Y_2, \ldots, Y_m)^T$, is either complete or is such that the set of possible modes is the whole sample space, and if $\varphi(\psi; a)$ is positive everywhere, then $A$ is respectively $G$ or $\tilde M$-ancillary for $\psi$ in the selective model, and adherence to the Conditionality Principle would require providing inference for the selected parameter conditionally on the data from the non-selected parameters.
This occurs, for example, if $Y_i\sim N(\theta_i, 1)$, $\theta_i \in \mathbb{R}$. 

This can be generalised to more interesting situations. Consider a situation where $Y\sim N(\theta, I_n)$ follows a $p$-dimensional normal distribution for some $p> 1$, and inference is provided for a linear function $\psi = \eta^T \theta$. This problem arises, for example, in regression settings with Gaussian errors and a fixed design, when interest lies on the coefficients of the projection parameter $\psi_s = [X(s)^TX(s)]^{-1}X(s)^T\mu$. In this case, if $P$ is a matrix containing $p-1$ linearly independent vectors orthogonal to $\eta$ as columns, then $A = P^T Y$ can be easily seen to be both $G$ and $\tilde{M}$-ancillary for $\psi$ without selection, and therefore also after selection. On defining $T = \eta^T Y$, one can easily show that the distribution of $T\mid A$ given selection is a truncated one-dimensional Gaussian distribution, which admits an exact analytical characterization provided the selection event is not too complicated. For selection events which can be expressed as a union of polyhedra, analytical characterisations are provided in \cite{leeetal}.

\section{Discussion}

Conditioning plays a fundamental role in all approaches to statistical inference. In selective inference, the choice of conditioning strategy is particularly important, as it can significantly impact the complexity of the problem and the validity of the results. We have discussed two ways to make this decision. The first relies on carefully assessing how information is processed during the selection stage. The second one exploits refined definitions of ancillarity that work nicely in selection problems.

Ancillarity in the presence of nuisance parameters is an intricate subject, and multiple definitions of it have been proposed, some of which are not preserved by conditioning on selection. A more complete examination of the relation between selection and ancillarity would be interesting.

\backmatter


\section*{Statements and declarations}

The authors have no conflicts of interest to disclose.

\bibliography{references}

\end{document}